\documentclass{fundam}

\usepackage{amsmath,amssymb}
\usepackage{hyperref}
\usepackage{mathabx}
\usepackage{stmaryrd}
\usepackage{txfonts}
\usepackage{orcidlink}
\usepackage[all]{xy}
\usepackage{kbordermatrix}

\newcommand{\kat}[1]{{\mathcal #1}}

\newcommand{\ideals}[2]{{\rm Id}(#1,#2)}
\newcommand{\unit}{\textsf{\bf 1}}

\newcommand{\trans}[1]{{#1}^{\smallsmile}}

\newcommand{\rres}[2]{#1/#2}

\newcommand{\id}{{\mathbb I}}

\newcommand{\up}[1]{{#1}^{\scriptstyle\uparrow}}
\newcommand{\down}[1]{{#1}^{\scriptstyle\downarrow}}
\newcommand{\downset}[1]{{\downarrow}#1}

\renewcommand{\kbldelim}{(}% Left delimiter
\renewcommand{\kbrdelim}{)}% Right delimiter

\allowdisplaybreaks

\begin{document}

\title{Arrow Operations in Categories of Lattice-valued Relations}

\author{Fatemeh Jowkar \orcidlink{0009-0002-3515-8708}\\
		Department of Mathematics,
		Yazd University,
		Yazd, Iran\\
		jokar.fatemeh1370@gmail.com
	\and
           Michael Winter \orcidlink{0000-0003-0847-0448}\\
           Department of Computer Science,
           Brock University, St.\ Catharines, ON, Canada \\
           mwinter@brocku.ca
          }
          
\runninghead{F. Jowkar, M. Winter}{Arrow Operations in Categories of Lattice-valued Relations}
          
\maketitle

\begin{abstract}
Arrow allegories provide a convenient abstract framework to work with lattice-valued relations, or more precisely, relations that use the elements of a given Heyting algebra as truth values. One characteristic of arrow allegories is that all relations 
of the given arrow allegory use the same Heyting algebra ${\mathcal H}$. In this paper we want to extend this approach to allegories where relations between different objects may use different lattices of truth values and 
even further to relations that use a different lattice of truth values for every pair in the relation. Therefore, we define three concrete allegories, $\mathrm{Rel}({\mathcal H})$, $\mathrm{Rel}^u({\mathcal H})$ and ${\mathcal H}{\rm-Rel}$,
where the allegory listed later is a full suballegory of the previous ones. These three allegories capture the three different situations mentioned above. In particular, ${\mathcal H}{\rm-Rel}$ is the standard example of an arrow category. 
We investigate these allegories and provide suitable categorical definitions for these structures.
\end{abstract}

\section{Introduction}

The concept of fuzzy sets \cite{Zadeh1965}, and more generally lattice valued fuzzy sets \cite{Goguen1967}, extends sets and, therefore, also relations to a many-valued context.
These sets and relations are used in a wide range of mathematical theories and applications including fuzzy controllers \cite{Mamdani1987}, linguistics \cite{deCock2000}, decision-making \cite{Kuzmin1982},
clustering \cite{Bezdek1978} and approximations and near sets \cite{JowkarLattices,JowkarNearFuzzy}. There are several algebraic theories covering different aspects of these sets and relations. 
For example, MV-algebras \cite{Cignoli1999,Hajek1998} are the algebraic semantics of \L{}ukasiewicz logic, i.e., the many-valued logic of \L{}ukasiewicz.

Relational methods, and allegories \cite{Freyd1990} in particular, provide a natural bridge between logic, order, and category theory.
In the classical setting, relations between sets form an allegory, and many logical operations can be expressed in relational terms.
Without referring to a specific logic, arrow allegories \cite{Winter2001,Winter2007} were introduced to capture the relational properties of many-valued relations. In this approach it is 
assumed that all relations of a given arrow allegory use the same Heyting algebra as truth values. In this paper we want to extend this approach to relations potentially using different lattices of truth values.
The first generalization allows relations between different objects to use different Heyting algebras. In a second step, this can be generalized to relations that use different Heyting algebras for each pair
of elements. We will define three allegories capturing these situations. Starting with the arrow allegory ${\mathcal H}{\rm-Rel}$ as the standard example of an allegory where all relations use ${\mathcal H}$
as truth values, we introduce $\mathrm{Rel}^u({\mathcal H})$ as the standard example of an allegory where relations between different objects use different Heyting algebras induced by $\mathcal H$, and,
finally, $\mathrm{Rel}({\mathcal H})$ as the standard example of an allegory where relations use different Heyting algebras induced by $\mathcal H$ for each pair of elements. It turns out that
${\mathcal H}{\rm-Rel}$ is a full suballegory of $\mathrm{Rel}^u({\mathcal H})$ which itself is a full suballegory of $\mathrm{Rel}({\mathcal H})$. We investigate these allegories and their properties
in order to provide suitable categorical definitions for these structures.

The paper is organized as follows.
In Section~2 we recall the necessary background on complete Heyting algebras and Heyting allegories. Furthermore, we introduce ideal relation as an abstract counterpart of the elements in ${\mathcal H}$
and show some basic properties thereof.
In Section~3 we first define the category $\mathrm{Rel}({\mathcal H})$ and describe its basic operations. Furthermore, we provide several counterexamples showing that the stated properties in $\mathrm{Rel}({\mathcal H})$
cannot be generalized even further. Based on these results we also introduce the two full suballegories $\mathrm{Rel}^u({\mathcal H})$ and ${\mathcal H}{\rm-Rel}$.
In Section~4 we provide suitable categorical definitions for all three allegories and show their basic properties.
Last but not least, in Section~5 we conclude and outline some future work.

\section{Preliminaries}

In this section we recall some basic properties of Heyting algebras and allegories.

A complete Heyting algebra ${\mathcal H}=(H,\wedge,\vee,\to,0,1)$ is a complete bounded lattice $(H,\wedge,\vee,0,1)$ so that
$$ z\leq x\to y\mbox{ iff }z\wedge x\leq y. $$
$x\to y$ is called the relative pseudo-complement of $x$ with respect to $y$
Note that every complete Heyting algebra is a distributive lattice, and that the first infinite distributive law
$$ x\wedge\bigvee_{i\in I}y_i=\bigvee_{i\in I}(x\wedge y_i) $$
is valid. For $x\in H$, we denote by $\downset{x}=\{ y \in H\mid y\leq x\}$ the down set 
of $x$. This is also known as the principal ideal generated by $x$.

\begin{lemma}
If ${\mathcal H}$ is a complete Heyting algebra and $z\in H$, then for an element $z\in H$ we define the following operation on $H$ by
$$ x\to_z y := (x\to y)\wedge z. $$
The complete sublattice $\downset{z}$ of ${\mathcal H}$ together with $x\to_z y$ is a Heyting algebra. We call the Heyting algebra 
$\downset{z}$ the subalgebra of ${\mathcal H}$ relative to $z$.
\end{lemma}

\begin{proof}
First of all $\downset{z}$ is closed under arbitrary unions and meets and contains $0$, i.e., $\downset{z}$ is a complete sublattice of ${\mathcal H}$ with greatest element $z$. 
Furthermore, $x\to_z y\leq z$, and, hence, $x\to_z y\in\downset{z}$. Now suppose $u\in\downset{z}$. Then we have
\begin{align*}
  u\leq x\to_z y
  &\iff u\leq (x\to y)\wedge z\\
  &\iff u\leq x\to y\mbox{ and }u\leq z\\  
  &\iff u\leq x\to y && \mbox{$u\in\downset{z}$}\\
  &\iff u\wedge x\leq y,
\end{align*}
i.e., $x\to_z y$ is the relative pseudo-complement of $x$ with respect to $y$ in $\downset{z}$.
\end{proof}

Please note that the operation $x\to_z y$ is defined for all elements of $H$. The previous lemma just shows that its restriction to $\downset{z}$ is a relative pseudo-complement in $\downset{z}$.

Arrow operations on truth values, i.e., on elements of a Heyting algebra ${\mathcal H}$ map elements of $H$ to the smallest and greatest element of the sublattice relative to a given
element $z$. Therefore, we define two operations $\rule{0pt}{6pt}^{\uparrow_x},\rule{0pt}{6pt}^{\downarrow_x}$ for an element $x\in H$ on $H$ by
\begin{align*}
	u^{\uparrow_x}
	&:= \begin{cases}x & \text{if }u\neq 0,\\ 0 & \text{otherwise},\end{cases}\\
	u^{\downarrow_x}
	&:= \begin{cases}x &\text{if }x\leq u,\\ 0 & \text{otherwise}.\end{cases}
\end{align*}
Please note that the operations $\rule{0pt}{6pt}^{\uparrow_x},\rule{0pt}{6pt}^{\downarrow_x}$ are again defined for all elements of $H$. Later in the definition
of the arrow operations in $\mathrm{Rel}({\mathcal H})$ both operations will be applied to elements that are in $\downset{x}$ only. This more general definition is
used because of the second statement in the following lemma.

\begin{lemma}\label{Lem:ArrowComponents}
The following properties are valid:
\begin{enumerate}
	\item $(\bigvee\limits_{i\in I}u_i)^{\uparrow_x}=\bigvee\limits_{i\in I}u_i^{\uparrow_x}$ for all $u_i\in\downset{x}$.
	\item $u^{\uparrow_x}\wedge y=u^{\uparrow_{x\wedge y}}$ for all $u,x,y\in H$.
	\item $(u\wedge v^{\downarrow_y})^{\uparrow_x}=u^{\uparrow_x}\wedge v^{\downarrow_y}$ for all $u\in\downset{x}$ and $v\in\downset{y}$ with $x\leq y$.
\end{enumerate}
\end{lemma}

\begin{proof}
\begin{enumerate}
	\item First, assume $\bigvee\limits_{i\in I}u_i=0$. Then $(\bigvee\limits_{i\in I}u_i)^{\uparrow_x}=0$ and we have $u_i=0$ for all $i\in I$. This implies
		$u_i^{\uparrow_x}=0$, and, hence, $\bigvee\limits_{i\in I}u_i^{\uparrow_x}=0$. Now, assume $\bigvee\limits_{i\in I}u_i\ne 0$. Then
		$(\bigvee\limits_{i\in I}u_i)^{\uparrow_x}=x$ and there exists an $i\in I$ so that $u_i\ne 0$. This implies $u_i^{\uparrow_x}=x$, and, hence, $\bigvee\limits_{i\in I}u_i^{\uparrow_x}=x$.
	\item If $u=0$, then both sides of the equation equal to $0$. If $u\ne 0$, then we have $u^{\uparrow_x}\wedge y=x\wedge y=u^{\uparrow_{x\wedge y}}$.
	\item If $u=0$ or $v^{\downarrow_y}=0$, then both sides of the equation equal to $0$. Now assume $u\ne 0$ and $v^{\downarrow_y}\ne 0$. Then $v^{\downarrow_y}=y$ since $v\leq y$, and 
	      we have $u\wedge y=u\ne 0$ since $u\leq x\leq y$ and $u\ne 0$. Now, we compute
		\begin{align*}
			(u\wedge v^{\downarrow_y})^{\uparrow_x}
			&= (u\wedge y)^{\uparrow_x} && \mbox{$v^{\downarrow_y}=y$}\\
			&= u^{\uparrow_x} && \mbox{$u\wedge y=u$}\\
			&= x && \mbox{$u\ne 0$}\\
			&= x\wedge y && \mbox{$x\leq y$}\\
			&= u^{\uparrow_x}\wedge y && \mbox{$u\ne 0$}\\
			&= u^{\uparrow_x}\wedge v^{\downarrow_y}. && \mbox{$v^{\downarrow_y}=y$}
		\end{align*}
\end{enumerate}
\end{proof}

\subsection{Heyting Allegories}

We denote the composition in a category by $;$. Please note that this operation has to be read from left to right, i.e., $f;g$ for two morphisms $f:A\to B$ and $g:B\to C$ means first $f$ and then $g$. We denote the
identity morphism on an object $A$ by $\id_A$. Heyting allegories provide a suitable categorical structure for relations.

\begin{definition}\label{Def:HeytingAllegory}
A Heyting allegory $\kat{R}$ is a category satisfying the following:
\begin{enumerate}
\item For all objects $A$ and $B$ the collection $\kat{R}[A,B]$ is a
      Heyting algebra. Meet, join, the induced ordering, relative pseudo-complement and the least and
      the greatest element are denoted by $\sqcap,\sqcup,\to,\sqsubseteq,\Bot_{AB},\Top_{AB}$, respectively.
\item There is a monotone operation $\trans{\rule{0pt}{6pt}}$
(called converse)
      mapping a relation $Q:A\to B$ to $\trans{Q}:B\to A$
      such that for all relations $Q:A\to B$ and $R:B\to C$ the following
      holds:
      $ \trans{(Q;R)}=\trans{R};\trans{Q}$ and
      $\trans{(\trans{Q})}=Q$.
\item For all relations $Q:A\to B, R:B\to C$ and $S:A\to C$ the
modular law
      $(Q;R)\sqcap S\sqsubseteq Q;(R\sqcap(\trans{Q};S))$
      holds.
\item For all relations $R:B\to C$ and $S:A\to C$ there is a
relation $S/R:A\to B$
      (called the left residual of $S$ and $R$) such that for all $X:A\to B$ the
      following holds:
      $X;R\sqsubseteq S\iff X\sqsubseteq \rres{S}{R}.$
\end{enumerate}
\end{definition}

In the following we assume that composition binds tighter than the lattice theoretic operations so that the modular law becomes
$Q;R\sqcap S\sqsubseteq Q;(R\sqcap\trans{Q};S)$.

The following lemma lists a few basic properties of Heyting allegories that we will need throughout the paper. A proof can be found 
in \cite{Winter2007}.

\begin{lemma}\label{Lem:Basics}
Let $\kat{R}$ be a Heyting allegory. Then we have:
\begin{enumerate}
	\item $\Top_{AA};\Top_{AB}=\Top_{AB};\Top_{BB}=\Top_{AB}$ for all objects $A,B$.
	\item $(Q\sqcap R;\Top_{BC});S=Q;S\sqcap R;\Top_{BD}$ for all $Q:A\to C, R:A\to B$ and $S:C\to D$.
	\item $Q;\Top_{BD}\sqcap \Top_{AC};R = Q;\Top_{BC};R$ for all $Q:A\to B$ and $R:C\to D$.
	\item $Q\sqsubseteq Q;\trans{Q};Q$ for all $Q:A\to B$.
\end{enumerate}
\end{lemma}

An important property of relations is totality. A relation $Q:A\to B$ is called total iff $\id_A\sqsubseteq Q;\trans{Q}$. Please note that 
total relations satisfy $Q;\Top_{BC}=\Top_{AC}$ for all objects $C$.

An abstract version of a singleton set as an object is given by the notion of a unit.

\begin{definition}
An object $\unit$ of a Heyting allegory is called a unit iff $\Top_{\unit\unit}=\id_\unit$ and $\Top_{A\unit}$ is total for every object $A$.
\end{definition}

\begin{lemma}\label{Lem:TopUnit}
Let $\kat{R}$ be a Heyting allegory with a unit. Then we have $\Top_{AB}=\Top_{A\unit};\Top_{\unit B}$.
\end{lemma}

\begin{proof}
This follows immediately from the totality of $\Top_{A\unit}$.
\end{proof}

An important property of Heyting allegories is uniformity. 

\begin{definition}\label{Def:uniform}
A Heyting allegory is called uniform iff $\Top_{AB};\Top_{BC}=\Top_{AC}$ for all objects $A,B$ and $C$.
\end{definition}

The following lemma provides an alternative definition of uniformity by quantifying over one resp.\ two objects only.

\begin{lemma}\label{Lem:uniform}
Let $\kat{R}$ be a Heyting allegory with a unit. Then the following statements are equivalent:
\begin{enumerate}
	\item $\kat{R}$ is uniform.
	\item $\Top_{AB};\Top_{BA}=\Top_{AA}$ for all objects $A$ and $B$.
	\item $\Top_{\unit A};\Top_{A\unit}=\Top_{\unit\unit}$ for all objects $A$.
\end{enumerate}
\end{lemma}

\begin{proof}
The implications $(1)\Rightarrow(2)$ and $(2)\Rightarrow(3)$ are trivial because the conclusion is each time a special case of the assumption. Now, suppose 
$\Top_{\unit A};\Top_{A\unit}=\Top_{\unit\unit}$ for all objects $A$. Then we have
\begin{align*}
	\Top_{AC}
	&= \Top_{AA};\Top_{AC} && \mbox{Lemma \ref{Lem:Basics}(1)}\\
	&= \Top_{A\unit};\Top_{\unit A};\Top_{AC} && \mbox{Lemma \ref{Lem:TopUnit}}\\	
	&= \Top_{A\unit};\Top_{\unit\unit};\Top_{\unit A};\Top_{AC} && \mbox{Lemma \ref{Lem:Basics}(1)}\\
	&= \Top_{A\unit};\Top_{\unit B};\Top_{B\unit};\Top_{\unit A};\Top_{AC} && \mbox{assumption}\\
	&= \Top_{AB};\Top_{BA};\Top_{AC} && \mbox{Lemma \ref{Lem:TopUnit}}\\
	&\sqsubseteq \Top_{AB};\Top_{BC},
\end{align*}
verifying that $\kat{R}$ is uniform.
\end{proof}

For concrete lattice valued relations, i.e., for relations between two sets $A$ and $B$ given by their characteristic function $Q:A\times B\to H$ an ideal relation
is a constant valued function. Therefore, any concrete ideal relation $Q$ corresponds to the lattice value $Q(a,b)$ for any $a\in A$ and $b\in B$. In other words, the collection
of ideal relations from $A$ to $B$ represents the lattice $L$ of truth values used by the relations from $A$ to $B$.  

\begin{definition}\label{Def:Ideal}
A relation $X:A\to B$ of a Heyting allegory $\kat{R}$ is called an ideal relation (or ideal for short) iff $\Top_{AA};X;\Top_{BB}=X$. We denote
by $\ideals{A}{B}$ the set of all ideals from $A$ to $B$.
\end{definition}

Notice that all relations on the unit, i.e., every relation $X:\unit\to\unit$, are ideals since $\Top_{\unit\unit}=\id_\unit$ and we have $\ideals{\unit}{\unit}=\kat{R}[\unit,\unit]$.
Since $X\sqsubseteq\Top_{AA};X;\Top_{BB}$ for all relations $X:A\to B$ it is sufficient to show that $\Top_{AA};X;\Top_{BB}\sqsubseteq X$ in order to verify that $X$ is an ideal.

\begin{lemma}\label{Lem:IdealProps}
Let $U,V:A\to A, X,Y : A\to B$ and $Z:B\to C$ be ideals. Then the relations $\Bot_{AB}, \Top_{AB}, \trans{X}, X\sqcap Y, X\sqcup Y, X\to Y$ and $X;Z$ are ideals and we
have $U;V=U\sqcap V$.
\end{lemma}

\begin{proof}
The cases $\Bot_{AB}, \trans{X}, X\sqcup Y$ and $X;Z$ are straightforward and left to the reader. The case $X\sqcap Y$ follows immediately from Lemma \ref{Lem:Basics}(2).

Now consider the following computation
\begin{align*}
  	\lefteqn{X\to Y\sqsubseteq X\to Y}\\
  	&\iff X\sqcap (X\to Y)\sqsubseteq Y\\
  	&~~\Longrightarrow \Top_{AA};(X\sqcap (X\to Y));\Top_{BB}\sqsubseteq \Top_{AA};Y;\Top_{BB}=Y && \mbox{$Y$ ideal}\\
  	&\iff \Top_{AA};(\Top_{AA};X;\Top_{BB}\sqcap (X\to Y));\Top_{BB}\sqsubseteq Y && \mbox{$X$ ideal}\\  	
  	&\iff \Top_{AA};X;\Top_{BB}\sqcap\Top_{AA};(X\to Y);\Top_{BB}\sqsubseteq Y && \mbox{Lemma \ref{Lem:Basics}(2)}\\  	
  	&\iff X\sqcap\Top_{AA};(X\to Y);\Top_{BB}\sqsubseteq Y && \mbox{$X$ ideal}\\  	
  	&\iff \Top_{AA};(X\to Y);\Top_{BB}\sqsubseteq X\to Y,
\end{align*}
verifying that $X\to Y$ is an ideal. For the remaining property we first show
\begin{align*}
	U;V
	&\sqsubseteq U;\Top_{AA}\\
	&= \Top_{AA};U;\Top_{AA};\Top_{AA} && \mbox{$U$ ideal}\\
	&= \Top_{AA};U;\Top_{AA} && \mbox{Lemma \ref{Lem:Basics}(1)}\\
	&= U. && \mbox{$U$ ideal}
\end{align*}
The inclusion $U;V\sqsubseteq V$ is shown analogously so that we conclude $U;V\sqsubseteq U\sqcap V$.
In addition, we have
\begin{align*}
  	U\sqcap V
  	&= \Top_{AA};U;\Top_{AA}\sqcap\Top_{AA};V;\Top_{AA} && \mbox{$U,V$ ideals}\\
  	&= \Top_{AA};U;\Top_{AA};V;\Top_{AA} && \mbox{Lemma \ref{Lem:Basics}(3)}\\  	
  	&= \Top_{AA};U;\Top_{AA};\Top_{AA};V;\Top_{AA} && \mbox{Lemma \ref{Lem:Basics}(1)}\\
  	&= U;V  && \mbox{$U,V$ ideals}
\end{align*}
verifying the opposite inclusion.
\end{proof}

The following theorem verifies that all lattices of truth values, i.e., the Heyting algebras $\ideals{A}{B}$, in a Heyting allegory are embedded in the Heyting algebra $\ideals{\unit}{\unit}$
as subalgebras relative to a specific element.

\begin{theorem}
The map $\varphi(X):\ideals{A}{B}\to\ideals{\unit}{\unit}$ between the set of ideal relations from $A$ to $B$ and the set of ideal relation on the unit
defined by 
$$ \varphi(X):=\Top_{\unit A};X;\Top_{B\unit} $$
is an isomorphism between the Heyting algebras $\ideals{A}{B}$ and the subalgebra of $\ideals{\unit}{\unit}$ relative to $\varphi(\Top_{AB})$. 
\end{theorem}

\begin{proof}
First of all, $\varphi(X)$ is well-defined since every relation on the unit is an ideal and
$\varphi(X)\sqsubseteq\varphi(\Top_{AB})$ for all $X\in\ideals{A}{B}$. Obviously, $\varphi$ maps $\Bot_{AB}$ to $\Bot_{\unit\unit}$ and $\Top_{AB}$ to $\varphi(\Top_{AB})$.
The following computations verify that $\varphi$ is a lattice homomorphism.
\begin{align*}
	\varphi(X\sqcup Y)
	&= \Top_{\unit A};(X\sqcup Y);\Top_{B\unit}\\
	&= \Top_{\unit A};X;;\Top_{B\unit}\sqcup \Top_{\unit A};Y;\Top_{B\unit}\\	
	&= \varphi(X)\sqcup\varphi(Y),\\
	\varphi(X\sqcap Y)
	&= \Top_{\unit A};(X\sqcap Y);\Top_{B\unit}\\	
	&= \Top_{\unit A};(\Top_{AA};X;\Top_{BB}\sqcap Y);\Top_{B\unit} && \mbox{$X$ ideal}\\
	&= \Top_{\unit A};X;\Top_{B\unit}\sqcap \Top_{\unit A};Y;\Top_{B\unit} && \mbox{Lemma \ref{Lem:Basics}(2)}\\	
	&= \varphi(X)\sqcap\varphi(Y).
\end{align*}

Now, we define $\psi(Z):=\Top_{A\unit};Z;\Top_{\unit B}$ for $Z:\unit\to\unit$ with $Z\sqsubseteq\varphi(\Top_{AB})$. Similar to 
$\varphi$ it can be shown that $\psi$ is also a lattice homomorphism. We obtain
\begin{align*}
  	\psi(\varphi(X))
  	&= \Top_{A\unit};\Top_{\unit A};X;\Top_{B\unit};\Top_{\unit B}\\
  	&= \Top_{AA};X;\Top_{BB} && \mbox{Lemma \ref{Lem:TopUnit}}\\
  	&= X. && \mbox{$X$ ideal}
\end{align*}
In addition, for ideals $Z\in\ideals{\unit}{\unit}$ with $Z\sqsubseteq\varphi(\Top_{AB})$ we have
\begin{align*}
	Z
	&= Z;Z;Z && \mbox{Lemma \ref{Lem:IdealProps}}\\
	&\sqsubseteq \varphi(\Top_{AB});Z;\varphi(\Top_{AB}) && \mbox{assumption}\\
	&= \Top_{\unit A};\Top_{AB};\Top_{B\unit};Z;\Top_{\unit A};\Top_{AB};\Top_{B\unit}\\
	&\sqsubseteq \Top_{\unit A};\Top_{A\unit};Z;\Top_{\unit B};\Top_{B\unit}.
\end{align*}

This implies for $X\in\ideals{A}{B}$
\begin{align*}
  	\lefteqn{Z\sqsubseteq\varphi(X)}\\
  	&\iff Z\sqsubseteq \Top_{\unit A};X;\Top_{B\unit}\\
  	&~~\Longrightarrow \Top_{A\unit};Z;\Top_{\unit B}\sqsubseteq\Top_{A\unit};\Top_{\unit A};X;\Top_{B\unit};\Top_{\unit B}\sqsubseteq\Top_{AA};X;\Top_{BB}\\
  	&\iff \psi(Z)\sqsubseteq X && \mbox{$X$ ideal}\\
	\lefteqn{\psi(Z)\sqsubseteq X}\\
	&\iff \Top_{A\unit};Z;\Top_{\unit B}\sqsubseteq X\\
	&~~\Longrightarrow \Top_{\unit A};\Top_{A\unit};Z;\Top_{\unit B};\Top_{B\unit}\sqsubseteq\Top_{\unit A};X;\Top_{B\unit}\\
	&~~\Longrightarrow Z\sqsubseteq\Top_{\unit A};X;\Top_{B\unit} && \mbox{see above}\\
	&\iff Z\sqsubseteq\varphi(X),
\end{align*}
i.e., that $\varphi$ and $\psi$ form a Galois connection. We obtain $Z\sqsubseteq\varphi(\psi(Z))$ since this is equivalent to $\psi(Z)\sqsubseteq\psi(Z)$ and
$\varphi(\psi(Z))=\Top_{\unit A};\Top_{A\unit};Z;\Top_{\unit B};\Top_{B\unit}\sqsubseteq\Top_{\unit\unit};Z;\Top_{\unit\unit}=Z$, i.e., $\varphi$ and $\psi$ are inverse to
each other, and, hence, $\varphi$ is bijective. It remains to show that 
$$ \varphi(X\to Y)=\varphi(X)\to_{\varphi(\Top_{AB})}\varphi(Y). $$

Therefore, assume $Z\sqsubseteq\varphi(\Top_{AB})$ and compute
\begin{align*}
	Z\sqsubseteq\varphi(X\to Y)
	&\iff \psi(Z)\sqsubseteq X\to Y && \mbox{Galois connection}\\
	&\iff \psi(Z)\sqcap X\sqsubseteq Y\\
	&\iff \psi(Z)\sqcap\psi(\varphi(X))\sqsubseteq Y && \mbox{see above}\\	
	&\iff \psi(Z\sqcap\varphi(X))\sqsubseteq Y && \mbox{$\psi$ lattice homomorphism}\\
	&\iff Z\sqcap\varphi(X)\sqsubseteq\varphi(Y) && \mbox{Galois connection}\\
	&\iff Z\sqsubseteq\varphi(X)\to\varphi(Y)\\	
	&\iff Z\sqsubseteq\varphi(X)\to_{\varphi(\Top_{AB})}\varphi(Y), && \mbox{$Z\sqsubseteq\varphi(\Top_{AB})$}
\end{align*}
verifying the desired property.
\end{proof}

Please note that if $\Top_{\unit A};\Top_{A\unit}=\Top_{\unit B};\Top_{B\unit}=\Top_{\unit\unit}$, then we have 
$\varphi(\Top_{AB})=\Top_{\unit\unit}$, i.e., the Heyting algebras $\ideals{\unit}{\unit}$ and $\ideals{A}{B}$ are isomorphic.
Consequently, if $\kat{R}$ is uniform, then the Heyting algebras of ideals $\ideals{A}{B}$ for all  objects $A$ and $B$ are isomorphic.

\section{The Heyting Allegory $\mathrm{Rel}({\mathcal H})$}

In this section we assume a fixed complete Heyting algebra ${\mathcal H}$ and introduce
the Heyting allegory $\mathrm{Rel}({\mathcal H})$. The objects of this allegory are
pairs $(A,\alpha)$ where $A$ is a set and $\alpha$ is a $H$-valued subset of $A$,
i.e., $\alpha$ is a function from $A$ to $H$. Given two objects $(A,\alpha)$ and $(B,\beta)$, 
we define $\alpha\times\beta:A\times B\to H$ by $(\alpha\times\beta)(a,b)=\alpha(a)\wedge\beta(b)$. Now
a morphism (relation) $Q:(A,\alpha)\to(B,\beta)$ is a $H$–valued relation $Q:A\times B\to H$ included in $\alpha\times\beta$.
For $Q:(A,\alpha)\to(B,\beta)$ and $R:(B,\beta)\to(C,\gamma)$ we define their composition by
$$ (Q;R)(a,c) := \bigvee_{b\in B}Q(a,b)\wedge R(b,c). $$

Please note that composition is well-defined since
\begin{align*}
	(Q;R)(a,c)
	&= \bigvee_{b\in B}Q(a,b)\wedge R(b,c)\\
	&\leq \bigvee_{b\in B}(\alpha(a)\wedge\beta(b))\wedge(\beta(b)\wedge\gamma(c))\\
	&\leq \bigvee_{b\in B}\alpha(a)\wedge\gamma(c)\\
	&= \alpha(a)\wedge\gamma(c)\\
	&= (\alpha\times\gamma)(a,c).
\end{align*}
The identity morphism on $(A,\alpha)$ is defined by
$$ \id_{(A,\alpha)}(a,a') := \left\{\begin{array}{ll} \alpha(a) &\mbox{if }a=a'\\ 0 &\mbox{otherwise}\end{array}\right. $$

Converse is defined as usual by $\trans{Q}(b,a):=Q(a,b)$ and all other constants and operations except the residual are defined
componentwise, i.e., we have
\begin{align*}
	(Q\sqcap R)(a,b) &:= Q(a,b)\wedge R(a,b),\\	
	(Q\sqcup R)(a,b) &:= Q(a,b)\vee R(a,b),\\	
	(Q\to R)(a,b) &:= Q(a,b)\to_{\alpha(a)\wedge\beta(b)} R(a,b),\\
      \Bot_{(A,\alpha)(B,\beta)}(a,b) &:= 0,\\
      \Top_{(A,\alpha)(B,\beta)}(a,b) &:= \alpha(a)\wedge\beta(b).
\end{align*}

Finally, the residual of $R:(B,\beta)\to(C,\gamma)$ and $S:(A,\alpha)\to(C,\gamma)$ is given by
$$ (\rres{S}{R})(a,b) := \bigwedge_{c\in C} R(b,c)\to_{\alpha(a)\wedge\beta(b)} S(a,c). $$

Please note that $\to_{\alpha(a)\wedge\beta(b)}$ is applied to elements in the definition of the residual above that are not necessarily in $\downset{(\alpha(a)\wedge\beta(b))}$.

In the following we will denote by $\overline{x}:A\to H$ for $x\in H$ the constant function that maps every element of the set $A$ to $x$.

\begin{theorem}\label{Th:RelHHeyting}
$\mathrm{Rel}({\mathcal H})$ is a Heyting allegory with unit $(\{\ast\},\overline{1})$.
\end{theorem}

\begin{proof}
The properties of all operations in $\mathrm{Rel}({\mathcal H})$ except the residual are shown essentially identical to the proof provided in Theorem~3.1 in \cite{Winter2007} and, therefore, omitted. For the residual, first assume that $X;R\sqsubseteq S$ for relations $X:(A,\alpha)\to(B,\beta), R:(B,\beta)\to(C,\gamma)$ and $S:(A,\alpha)\to(C,\gamma)$.
Then we have $X(a,b)\wedge R(b,c)\leq\bigvee\limits_{b\in B}X(a,b)\wedge R(b,c)=(X;R)(a,c)\leq S(a,c)$ for all $a\in A, b\in B$ and $c\in C$.
This implies $X(a,b)\leq R(b,c)\to S(a,c)$, and, hence, $X(a,b)\leq R(b,c)\to_{\alpha(a)\wedge\beta(b)}S(a,c)$ since $X(a,b)\leq\alpha(a)\wedge\beta(b)$.
We conclude $X(a,b)\leq\bigwedge\limits_{c\in C}R(b,c)\to_{\alpha(a)\wedge\beta(b)}S(a,c)=(\rres{S}{R})(a,b)$, i.e., $X\sqsubseteq\rres{S}{R}$.
For the opposite implication assume $X\sqsubseteq\rres{S}{R}$. Then we obtain
$X(a,b)\leq(\rres{S}{R})(a,b)=\bigwedge\limits_{c\in C}R(b,c)\to_{\alpha(a)\wedge\beta(b)}S(a,c)\leq R(b,c)\to_{\alpha(a)\wedge\beta(b)}S(a,c)\leq R(b,c)\to S(a,c)$
for all for all $a\in A, b\in B$ and $c\in C$. This implies $X(a,b)\wedge R(b,c)\leq S(a,c)$, and, hence, $(X;R)(a,c)=\bigvee\limits_{b\in B}X(a,b)\wedge R(b,c)\leq S(a,c)$,
i.e., $X;R\sqsubseteq S$.

In order to verify that $(\{\ast\},\overline{1})$ is a unit notice that the relation $\Top_{(\{\ast\},\overline{1})(\{\ast\},\overline{1})}$ is component-wise given by
$\Top_{(\{\ast\},\overline{1})(\{\ast\},\overline{1})}(\ast,\ast)=1$ and, therefore, identical to $\id_{(\{\ast\},\overline{1})}$. Furthermore,
\begin{align*}
	(\Top_{(A,\alpha)(\{\ast\},\overline{1})};\Top_{(\{\ast\},\overline{1})(A,\alpha)})(a,a)
	&= \Top_{(A,\alpha)(\{\ast\},\overline{1})}(a,\ast);\Top_{(\{\ast\},\overline{1})(A,\alpha)}(\ast,a)\\
	&= \alpha(a)\\
	&= \id_{(A,\alpha)}(a,a)
\end{align*}
verifying that $\Top_{(A,\alpha)(\{\ast\},\overline{1})}$ is total.
\end{proof}

We are going to define arrow operations in $\mathrm{Rel}({\mathcal H})$ as follows.

\begin{definition}\label{Def:RelHOperators}
For every relation $Q:(A,\alpha)\to(B,\beta)$ in $\mathrm{Rel}({\mathcal H})$ we define the arrow operations by
\begin{align*}
	\up{Q}(a,b)
	&:= (Q(a,b))^{\uparrow_{\alpha(a)\wedge\beta(b)}}\\
	\down{Q}(a,b)
	&:= (Q(a,b))^{\downarrow_{\alpha(a)\wedge\beta(b)}}
\end{align*}
\end{definition}

The following definition combines a Heyting allegory with the very basic properties of arrow operations. This definition serves as the general context in which we will
investigate arrow operations.

\begin{definition}\label{Def:PairArrows}
Let $\kat{R}$ be a Heyting allegory. A pair $(\up{\rule{0pt}{6pt}},\down{\rule{0pt}{6pt}})$ of operations is called a pair of arrow operations on
$\kat{R}$ iff we have:
\begin{enumerate}
    	\item $\up{Q},\down{Q}:A\to B$ for all $Q:A\to B$.
    	\item $(\up{\rule{0pt}{6pt}},\down{\rule{0pt}{6pt}})$ is a Galois correspondence, i.e., $\up{Q}\sqsubseteq R$ iff $Q\sqsubseteq\down{R}$ for all $Q,R: A\to B$.
    	\item $\up{(Q\sqcap\down{R})}=\up{Q}\sqcap\down{R}$ for all $Q,R:A\to B$.
    	\item $\up{\trans{Q}}=\trans{\up{Q}}$ for all $Q:A\to B$.
\end{enumerate}
\end{definition}

\begin{lemma}\label{Lem:PairArrowsBasics}
Let $\kat{R}$ be a Heyting allegory and $(\up{\rule{0pt}{6pt}},\down{\rule{0pt}{6pt}})$ of arrow operations on $\kat{R}$. Then we have:
\begin{enumerate}
	\item $Q\sqsubseteq Q^{\uparrow\downarrow}$ and $Q^{\downarrow\uparrow}\sqsubseteq Q$.
	\item $\down{\rule{0pt}{6pt}}$ and $\up{\rule{0pt}{6pt}}$ are monotone.
	\item $\up{Q}=Q^{\uparrow\downarrow\uparrow}$ and $\down{Q}=Q^{\downarrow\uparrow\downarrow}$.
	\item $\down{(Q\sqcap R)}=\down{Q}\sqcap\down{R}$ and $\up{(Q\sqcup R)}=\up{Q}\sqcup\up{R}$.
	\item $\up{\Bot}_{AB}=\Bot_{AB}$.
	\item $\trans{\down{Q}}=\down{\trans{Q}}$ for all $Q:A\to B$.
\end{enumerate}
\end{lemma}

\begin{proof}
\begin{enumerate}
	\item[{\rm 1.-4.}] These properties follow immediately from the fact that $\down{\rule{0pt}{6pt}}$ and $\up{\rule{0pt}{6pt}}$ is a Galois correspondence.
	\setcounter{enumi}{4}
	\item First of all, from $\Bot_{AB}\sqsubseteq\down{\Bot}_{AB}$ we obtain $\up{\Bot}_{AB}\sqsubseteq\Bot_{AB}$, and, hence,
	      $\up{\Bot}_{AB}=\Bot_{AB}$. 
    	\item For any $X:B\to A$ we have
    		\begin{align*}
    			X\sqsubseteq\trans{\down{Q}}
    			&\iff \trans{X}\sqsubseteq\down{Q}\\
    			&\iff \up{\trans{X}}\sqsubseteq Q && \mbox{Definition \ref{Def:PairArrows}(2)}\\
    			&\iff \trans{\up{X}}\sqsubseteq Q && \mbox{Definition \ref{Def:PairArrows}(4)}\\    			
    			&\iff \up{X}\sqsubseteq\trans{Q}\\    			
    			&\iff X\sqsubseteq\down{\trans{Q}}, && \mbox{Definition \ref{Def:PairArrows}(2)}
    		\end{align*}
    		from which we conclude $\trans{\down{Q}}=\down{\trans{Q}}$.
\end{enumerate}
\end{proof}

The arrow operations on $\mathrm{Rel}({\mathcal H})$ defined above constitute a pair of arrow operations as the following theorem verifies.

\begin{theorem}\label{Th:RelHArrowPair}
The  pair $(\up{\rule{0pt}{6pt}},\down{\rule{0pt}{6pt}})$ is a pair of arrow operations on $\mathrm{Rel}({\mathcal H})$.
\end{theorem}

\begin{proof}
\begin{enumerate}
	\item This follows immediately from the definition of the arrow operations.
	\item Assume $\up{Q}\sqsubseteq R$ and distinguish two cases. If $\up{Q}(a,b)=0$, then $Q(a,b)=\up{Q}(a,b)=0\leq \down{R}(a,b)$. If $\up{Q}(a,b)=\alpha(a)\wedge\beta(b)$,
		then we have $R(a,b)=\alpha(a)\wedge\beta(b)$, and, hence, $Q(a,b)\leq\alpha(a)\wedge\beta(b)=R(a,b)=\down{R}(a,b)$. For the opposite implication
		assume $Q\sqsubseteq\down{R}$ and distinguish again two cases. If $\down{R}(a,b)=\alpha(a)\wedge\beta(b)$, then $\up{Q}(a,b)\leq\alpha(a)\wedge\beta(b)= R(a,b)$.
		If $\down{R}(a,b)=0$, then $Q(a,b)=0$, and, hence, $\up{Q}(a,b)=0\leq R(a,b)$. 
	\item We compute
		\begin{align*}
			\up{(Q\sqcap\down{R})}(a,b)
			&= ((Q\sqcap\down{R})(a,b))^{\uparrow_{\alpha(a)\wedge\beta(b)}}\\
			&= (Q(a,b)\wedge\down{R}(a,b))^{\uparrow_{\alpha(a)\wedge\beta(b)}}\\	
			&= (Q(a,b)\wedge(R(a,b))^{\downarrow_{\alpha(a)\wedge\beta(b)}})^{\uparrow_{\alpha(a)\wedge\beta(b)}}\\
			&= (Q(a,b))^{\uparrow_{\alpha(a)\wedge\beta(b)}}\wedge(R(a,b))^{\downarrow_{\alpha(a)\wedge\beta(b)}} && \mbox{Lemma \ref{Lem:ArrowComponents}(2)}\\
			&= \up{Q}(a,b)\wedge\down{R}(a,b)\\
			&= (\up{Q}\sqcap\down{R})(a,b).
		\end{align*}
	\item This is shown by
		\begin{align*}
			\up{\trans{Q}}(b,a)
			&= (\trans{Q}(b,a))^{\uparrow_{\beta(b)\wedge\alpha(a)}}\\
			&=  (Q(a,b))^{\uparrow_{\alpha(a)\wedge\beta(b)}}\\
			&= \up{Q}(a,b)\\
			&= \trans{\up{Q}}(b,a).
		\end{align*}
\end{enumerate}
\end{proof}

With the next definition we want to cover the situation that all relations on an object $A$ use the same Heyting algebra for every pair of elements. Intuitively, this is achieved by requiring that
every relation $R:A\to A$ restricted to a subset of $A$ uses the same Heyting algebra than $R$ itself.

\begin{definition}
Let $\kat{R}$ be a Heyting allegory with a unit $\unit$ and a pair of arrow operation $\up{\rule{0pt}{6pt}}$on $\kat{R}$. Then an object $A$ is
called uniform iff for all non-zero relations $v:\unit\to A$ we have 
$$ \up{v};\Top_{A\unit}=\Top_{\unit A};\Top_{A\unit}. $$
Furthermore, we call $\kat{R}$ locally uniform iff all objects are uniform.
\end{definition}

The following lemma shows that the notion of a uniform object in $\mathrm{Rel}({\mathcal H})$ indeed covers the intuition mentioned above. To this end we will use the abbreviation 
${\rm sup}~\alpha:=\bigvee\limits_{a\in A}\alpha(a)$ in the remainder of the paper.

\begin{lemma}\label{Lem:RelHobjectuniform}
An object $(A,\alpha)$ in $\mathrm{Rel}({\mathcal H})$ is uniform iff $\alpha=\overline{{\rm sup}~\alpha}$.
\end{lemma}

\begin{proof}
First we compute 
\begin{align*}
	(\Top_{(\{\ast\},\overline{1})(A,\alpha)};\Top_{(A,\alpha)(\{\ast\},\overline{1})})(\ast,\ast)
	&= \bigvee_{a\in A}\Top_{(\{\ast\},\overline{1})(A,\alpha)}(\ast,a)\wedge\Top_{(A,\alpha)(\{\ast\},\overline{1})}(a,\ast)\\
	&= \bigvee_{a\in A}\alpha(a)\\
	&= {\rm sup}~\alpha.
\end{align*}
Now, assume that $(A,\alpha)$ is uniform. Then for a given $a\in A$ we define a relation $v_a:(\{\ast\},\overline{1})\to(A,\alpha)$ by
$$ v_a(\ast,a') := \begin{cases}\alpha(a) &\text{if }a=a',\\0 & \text{otherwise}.\end{cases} $$
$v_a$ is a non-zero relation and we have $\up{v}_a=v_a$. We obtain
\begin{align*}
	(\up{v}_a;\Top_{(A,\alpha)(\{\ast\},\overline{1})})(\ast,\ast)
	&= (v_a;\Top_{(A,\alpha)(\{\ast\},\overline{1})})(\ast,\ast)\\
	&= \bigvee_{a'\in A}v_a(\ast,a')\wedge\Top_{(A,\alpha)(\{\ast\},\overline{1})}(a',\ast)\\	
	&= \bigvee_{a'\in A}v_a(\ast,a')\wedge\alpha(a')\\	
	&= \alpha(a)
\end{align*}

The calculations above show that $\alpha(a)={\rm sup}~\alpha$ for all $a\in A$, i.e., that $\alpha=\overline{{\rm sup}~\alpha}$. For the
opposite implication assume $\alpha=\overline{{\rm sup}~\alpha}$ and a non-zero relation $v:(\{\ast\},\overline{1})\to(A,\alpha)$. Then we have 
$v(\ast,a)\leq\alpha(a)={\rm sup}~\alpha$ and $v(\ast,a)\ne 0$ for all elements $a$ from a non-empty set $M\subseteq A$.
The latter implies that $\up{v}(\ast,a)={\rm sup}~\alpha$ for all $a\in M$. We obtain
\begin{align*}
	(\up{v};\Top_{(A,\alpha)(\{\ast\},\overline{1})})(\ast,\ast)
	&= \bigvee_{a\in A}\up{v}(\ast,a)\wedge\Top_{(A,\alpha)(\{\ast\},\overline{1})}(a,\ast)\\	
	&= \bigvee_{a\in M}\up{v}(\ast,a)\wedge\alpha(a)\\	
	&= \bigvee_{a\in M}{\rm sup}~\alpha && \mbox{$\alpha=\overline{{\rm sup}~\alpha}$}\\
	&= {\rm sup}~\alpha\\
	&= (\Top_{(\{\ast\},\overline{1})(A,\alpha)};\Top_{(A,\alpha)(\{\ast\},\overline{1})})(\ast,\ast).
\end{align*}
\end{proof}

Now we are ready to show some additional properties of the arrow operations in $\mathrm{Rel}({\mathcal H})$.

\begin{theorem}\label{Th:RelHArrow}
The following properties are valid in $\mathrm{Rel}({\mathcal H})$:
\begin{enumerate}
	\item $\up{\id}_{(A,\alpha)}=\id_{(A,\alpha)}$ and $\up{\Top}_{(A,\alpha)(B,\beta)}=\Top_{(A,\alpha)(B,\beta)}$ for all objects $(A,\alpha)$ and $(B,\beta)$.
    	\item $\up{(Q;\down{R})}\sqcap\Top_{(A,\alpha)(B,\beta)};\Top_{(B,\beta)(C,\gamma)}=(\up{Q}\sqcap\Top_{(A,\alpha)(C,\gamma)};\Top_{(C,\gamma)(B,\beta)});\down{R}$ 
    		for all $Q:(A,\alpha)\to(B,\beta)$ and $R:(B,\beta)\to(C,\gamma)$ and uniform objects $(B,\beta)$ and $(C,\gamma)$.
    	\item If $Q\ne\Bot_{(\{\ast\},\overline{1})(\{\ast\},\overline{1})}$, then we have $\up{Q}=\Top_{(\{\ast\},\overline{1})(\{\ast\},\overline{1})}$.
\end{enumerate}
\end{theorem}

\begin{proof}
\begin{enumerate}
	\item This is shown by
		\begin{align*}
			\up{\id}_{(A,\alpha)}(a_1,a_2)
			&= (\id_{(A,\alpha)}(a_1,a_2))^{\uparrow_{\alpha(a_1)\wedge\alpha(a_2)}}\\
			&= \begin{cases}\alpha(a_1)^{\uparrow_{\alpha(a_1)}} &\text{if }a_1=a_2\\0^{\uparrow_{\alpha(a_1)\wedge\alpha(a_2)}}&\text{otherwise}\end{cases}\\
			&= \begin{cases}\alpha(a_1) &\text{if }a_1=a_2\\0 &\text{otherwise}\end{cases}\\
			&= \id_{(A,\alpha)}(a_1,a_2),\\
			\up{\Top}_{(A,\alpha)(B,\beta)}(a,b)
			&= (\Top_{(A,\alpha)(B,\beta)}(a,b))^{\uparrow_{\alpha(a)\wedge\beta(b)}}\\			
			&= (\alpha(a)\wedge\beta(b))^{\uparrow_{\alpha(a)\wedge\beta(b)}}\\
			&= \alpha(a)\wedge\beta(b)\\
			&= \Top_{(A,\alpha)(B,\beta)}(a,b).
		\end{align*}
	\item Assume $Q:(A,\alpha)\to(B,\beta)$, $R:(B,\beta)\to(C,\gamma)$, and $(B,\beta)$ and $(C,\gamma)$ are uniform. From Lemma \ref{Lem:RelHobjectuniform}
		we get $\beta=\overline{{\rm sup}~\beta}$ and $\gamma=\overline{{\rm sup}~\gamma}$. First we obtain for all $S:(D,\delta)\to(E,\eta)$ and $(F,\phi)$
		\begin{align*}
			\lefteqn{(S\sqcap\Top_{(D,\delta)(F,\phi)};\Top_{(F,\phi)(E,\eta)})(d,e)}\\
			&= S(d,e)\wedge(\Top_{(D,\delta)(F,\phi)};\Top_{(F,\phi)(E,\eta)})(d,e)\\
			&= S(d,e)\wedge\bigvee_{f\in F}\Top_{(D,\delta)(F,\phi)}(d,f)\wedge\Top_{(F,\phi)(E,\eta)})(f,e)\\
			&= S(d,e)\wedge\bigvee_{f\in F}\delta(d)\wedge\phi(f)\wedge\eta(e)\\	
			&= S(d,e)\wedge\delta(d)\wedge\eta(e)\wedge{\rm sup}~\phi\\
			&= S(d,e)\wedge{\rm sup}~\phi,
		\end{align*}
		where the last equation follows from $S(d,e)\leq\delta(d)\wedge\eta(e)$. Now we compute
		\begin{align*}
			\lefteqn{(\up{(Q;\down{R})}\sqcap\Top_{(A,\alpha)(B,\beta)};\Top_{(B,\beta)(C,\gamma)})(a,c)}\\
			&= (\up{(Q;\down{R})})(a,c)\wedge{\rm sup}~\beta && \mbox{see above}\\
			&= ((Q;\down{R})(a,c))^{\uparrow_{\alpha(a)\wedge\gamma(c)}}\wedge{\rm sup}~\beta\\
			&= ((Q;\down{R})(a,c))^{\uparrow_{\alpha(a)\wedge{\rm sup}\,\beta\wedge\gamma(c)}} && \mbox{Lemma~\ref{Lem:ArrowComponents}(2)}\\
			&= ((Q;\down{R})(a,c))^{\uparrow_{\alpha(a)\wedge{\rm sup}\,\beta\wedge{\rm sup}\,\gamma}} && \mbox{$\gamma=\overline{{\rm sup}~\gamma}$}\\
			&= \left(\bigvee_{b\in B}Q(a,b)\wedge\down{R}(b,c)\right)^{\uparrow_{\alpha(a)\wedge{\rm sup}\,\beta\wedge{\rm sup}\,\gamma}}\\
			&= \left(\bigvee_{b\in B}Q(a,b)\wedge(R(b,c))^{\downarrow_{\beta(b)\wedge\gamma(c)}}\right)^{\uparrow_{\alpha(a)\wedge{\rm sup}\,\beta\wedge{\rm sup}\,\gamma}}\\
			&= \left(\bigvee_{b\in B}Q(a,b)\wedge(R(b,c))^{\downarrow_{{\rm sup}\,\beta\wedge{\rm sup}\,\gamma}}\right)^{\uparrow_{\alpha(a)\wedge{\rm sup}\,\beta\wedge{\rm sup}\,\gamma}}
				&& \mbox{$\beta=\overline{{\rm sup}~\beta}, \gamma=\overline{{\rm sup}~\gamma}$}\\
			&= \bigvee_{b\in B}(Q(a,b)\wedge(R(b,c))^{\downarrow_{{\rm sup}\,\beta\wedge{\rm sup}\,\gamma}})^{\uparrow_{\alpha(a)\wedge{\rm sup}\,\beta\wedge{\rm sup}\,\gamma}}
				&& \mbox{Lemma~\ref{Lem:ArrowComponents}(1)}\\
			&= \bigvee_{b\in B}(Q(a,b))^{\uparrow_{\alpha(a)\wedge{\rm sup}\,\beta\wedge{\rm sup}\,\gamma}}\wedge(R(b,c))^{\downarrow_{{\rm sup}\,\beta\wedge{\rm sup}\,\gamma}}
				&& \mbox{Lemma~\ref{Lem:ArrowComponents}(3)}\\	
		      &= \bigvee_{b\in B}(Q(a,b))^{\uparrow_{\alpha(a)\wedge{\rm sup}\,\beta}}\wedge{\rm sup}~\gamma\wedge(R(b,c))^{\downarrow_{{\rm sup}\,\beta\wedge{\rm sup}\,\gamma}}
				&& \mbox{Lemma~\ref{Lem:ArrowComponents}(2)}\\
			&= \bigvee_{b\in B}(Q(a,b))^{\uparrow_{\alpha(a)\wedge\beta(b)}}\wedge{\rm sup}~\gamma\wedge(R(b,c))^{\downarrow_{\beta(b)\wedge\gamma(c)}}
				&& \mbox{$\beta=\overline{{\rm sup}~\beta}, \gamma=\overline{{\rm sup}~\gamma}$}\\
			&= \bigvee_{b\in B}\up{Q}(a,b)\wedge{\rm sup}~\gamma\wedge\down{R}(b,c)\\
			&= \bigvee_{b\in B}(\up{Q}\sqcap\Top_{(A,\alpha)(C,\gamma)};\Top_{(C,\gamma)(B,\beta)})(a,b)\wedge\down{R}(b,c) && \mbox{see above}\\
			&= ((\up{Q}\sqcap\Top_{(A,\alpha)(C,\gamma)};\Top_{(C,\gamma)(B,\beta)});\down{R})(a,c).
		\end{align*}
	\item Assume $Q\ne\Bot_{(\{\ast\},\overline{1})(\{\ast\},\overline{1})}$. Because $Q$ is a relation on a singleton set and $\overline{1}$ is a constant function, this follows immediately from the 
		definition of $\rule{0pt}{6pt}^{\uparrow_1}$.
\end{enumerate}
\end{proof}

We want to mention two full suballegories of $\mathrm{Rel}({\mathcal H})$. We define $\mathrm{Rel}^u({\mathcal H})$ to be the full
suballegory of $\mathrm{Rel}({\mathcal H})$ induced by all uniform objects of $\mathrm{Rel}({\mathcal H})$ and ${\mathcal H}{\rm-Rel}$ to be the full suballegory
of $\mathrm{Rel}({\mathcal H})$ (and/or $\mathrm{Rel}^u({\mathcal H})$) induced by all objects that have the function $\overline{1}$ for $\alpha$. Please note
that ${\mathcal H}{\rm-Rel}$ is exactly the arrow allegory ${\mathcal L}{\rm-Rel}$ in \cite{Winter2007}. These three allegories establish a hierarchy of allegories modeling three different levels of ''fuzziness''. In ${\mathcal H}{\rm-Rel}$ all relations use the same lattice of truth values, i.e., this allegory can be seen as the standard example of an arrow allegory. In $\mathrm{Rel}^u({\mathcal H})$ relations on different objects may use different lattices but all pairs within a single relation do use the same lattice of truth values. And finally in $\mathrm{Rel}({\mathcal H})$ every pair in each relation might use a different lattice. In the next section we are interested in characterizing these three situations abstractly by providing suitable axioms for each situation.

In the remainder of this section we will provide examples illustrating the necessity of the assumptions in Theorem~\ref{Th:RelHArrow}.
We will present those examples by using matrix representations of relations between finite objects in $\mathrm{Rel}(\mathcal H)$.
This notation allows explicit computations to be presented in a compact and transparent way and is particularly useful for
illustrating properties in $\mathrm{Rel}(\mathcal H)$. Throughout these examples we use the distributive lattice ${\mathcal H}$ whose Hasse diagram is shown in
Figure~\ref{fig:HeytingExample}. Please note that since ${\mathcal H}$ is finite it is indeed a complete Heyting algebra.

\begin{figure}[h]
\[\xymatrix@R=0.7em@C=0.7em{
 && 1\\
 &c\ar@{-}[ru]&&d\ar@{-}[lu]\\
 a\ar@{-}[ru]&&b\ar@{-}[lu]\ar@{-}[ru]\\
 &0\ar@{-}[lu]\ar@{-}[ru]
}\]
\caption{The finite Heyting algebra used in the following examples.}
\label{fig:HeytingExample}
\end{figure}
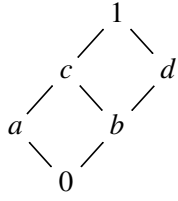

Suppose $(A,\alpha)$ and $(B,\beta)$ are objects of $\mathrm{Rel}(\mathcal H)$ so that
$A$ and $B$ are finite and enumerated, i.e., we have
\[
A=\{a_1,\ldots,a_m\},\qquad B=\{b_1,\ldots,b_n\}.
\]
Then a relation $Q:(A,\alpha)\longrightarrow(B,\beta)$ is completely determined by the following set of triples of values from $\mathcal{H}$
\[
\alpha(a_i),~\beta(b_j),~Q(a_i,b_j), \qquad 1\le i\le m,\; 1\le j\le n,
\]
and therefore may be represented by the labeled matrix
\[
  Q = \kbordermatrix{
    & \beta(b_1) & {\displaystyle \cdots} & \beta(b_n) \\
    \alpha(a_1) & Q(a_1,b_1) & \cdots & Q(a_1,b_n) \\
    \vdots & \vdots & \ddots & \vdots \\
    \alpha(a_m) & Q(a_m,b_1) & \cdots & Q(a_m,b_n) 
  }
\]
Please note that every entry $Q(a_i,b_j)$ in such a matrix must be smaller or equal to both the corresponding row label $\alpha(a_i)$ and column label $\beta(b_j)$. Furthermore, composition of two relations
given in matrix form can be performed if the column labels of the first matrix equal the row labels of the second matrix. In that case the composition can be 
calculated by regular matrix multiplication known from linear algebra with meet ($\sqcap$) instead of multiplication and join ($\sqcup$) instead of summation.
For a concrete example, let
\begin{align*}
	A=\{r,s\},\qquad\alpha(r)=1,\;\alpha(s)=1,\\
      	B=\{u,v\},\qquad\beta(u)=1,\;\beta(v)=c,\\
      	C=\{x,y\},\qquad\gamma(x)=d,\;\gamma(y)=1.
\end{align*}
Then the relations $Q:(A,\alpha)\to(B,\beta)$ and $R:(B,\beta)\to(C,\gamma)$ defined by
\begin{align*}
	Q(r,u)=a, Q(r,v)=0, Q(s,u)=0, Q(s,v)=c,\\
	R(u,x)=d, R(u,y)=0, R(v,x)=0, R(v,y)=c
\end{align*}
may be represented in matrix form by
\[
  Q = \kbordermatrix{
    & 1 & c \\
    1 & a & 0\\
    1 & 0  & c 
  },\qquad
  R = \kbordermatrix{
    & d & 1 \\
    1 & d & 0\\
    c & 0  & c 
  },
\]
and the composition $Q;R$ by
\[
Q;R=\kbordermatrix{
    & d & 1 \\
    1 & 0 & 0\\
    1 & 0  & c 
  }.
\]
Our first two examples will show that the uniformity assumptions in Theorem~\ref{Th:RelHArrow}(2) cannot be dropped.
Therefore, we define $\gamma'$ as the constant function induced by $1$, i.e., $\gamma'=\overline{1}$. Then we define 
$R_1:(B,\beta)\to(C,\gamma')$ by using the same entries in the matrix as for $R$. This is well-defined since $\gamma$ is a smaller than or
equal to $\gamma'$ for all values of $C$. In this case $(C,\gamma')$ is uniform but $(B,\beta)$ is not. We obtain for the left-hand
side of Theorem~\ref{Th:RelHArrow}(2)
\begin{align*}
	\lefteqn{\up{(Q;\down{R}_1)}\sqcap\Top_{(A,\alpha)(B,\beta)};\Top_{(B,\beta)(C,\gamma')}}\\
	&= \up{\left(\kbordermatrix{
	    & 1 & c \\
	    1 & a & 0\\
	    1 & 0  & c 
	  };
	  \kbordermatrix{
	    & 1 & 1 \\
	    1 & 0 & 0\\
	    c & 0  & c 
        }\right)}
        \sqcap\kbordermatrix{
	    & 1 & c \\
	    1 & 1 & c\\
	    1 & 1 & c 
        };\kbordermatrix{
	    & 1 & 1 \\
	    1 & 1 & 1\\
	    c & c & c 
        }\\
        &= \renewcommand{\kbrdelim}{)^\uparrow}\kbordermatrix{
	    & 1 & 1 \\
	    1 & 0 & 0\\
	    1 & 0 & c 
        }\renewcommand{\kbrdelim}{)}
        \sqcap\kbordermatrix{
	    & 1 & 1 \\
	    1 & 1 & 1\\
	    1 & 1 & 1 
        }\\
        &= \kbordermatrix{
	    & 1 & 1 \\
	    1 & 0 & 0\\
	    1 & 0 & 1
        }.
\end{align*}
For the right-hand side of Theorem~\ref{Th:RelHArrow}(2) we get
\begin{align*}
	\lefteqn{(\up{Q}\sqcap\Top_{(A,\alpha)(C,\gamma')};\Top_{(C,\gamma')(B,\beta)});\down{R}_1}\\
	&= \left(\kbordermatrix{
	    & 1 & c \\
	    1 & 1 & 0\\
	    1 & 0  & c 
	  }\sqcap\kbordermatrix{
	    & 1 & 1 \\
	    1 & 1 & 1\\
	    1 & 1 & 1 
        };\kbordermatrix{
	    & 1 & c \\
	    1 & 1 & 1\\
	    1 & c & c 
        }\right); \kbordermatrix{
	    & 1 & 1 \\
	    1 & 0 & 0\\
	    c & 0  & c 
        }\\	
        &= \kbordermatrix{
	    & 1 & c \\
	    1 & 1 & 0\\
	    1 & 0  & c 
	  }; \kbordermatrix{
	    & 1 & 1 \\
	    1 & 0 & 0\\
	    c & 0  & c 
        }\\	
        &= \kbordermatrix{
	    & 1 & 1 \\
	    1 & 0 & 0\\
	    c & 0 & c 
        }, 
\end{align*}
verifying that equation from Theorem~\ref{Th:RelHArrow}(2) is not valid. Similar to above
we define $\beta'$ as the constant function induced by $1$, i.e., $\beta'=\overline{1}$, and relation $Q_2:(A,\alpha)\to(B,\beta')$ and $R_2:(B,\beta')\to(C,\gamma)$
by using the same entries in the matrix as for $Q$ and $R$, respectively. Again, this is well-defined since $\beta$ is a smaller than or
equal to $\beta'$ for all values of $B$. This time $(B,\beta')$ is uniform but $(C,\gamma)$ is not. We obtain for the left-hand
side of Theorem~\ref{Th:RelHArrow}(2)
\begin{align*}
	\lefteqn{\up{(Q_2;\down{R}_2)}\sqcap\Top_{(A,\alpha)(B,\beta')};\Top_{(B,\beta')(C,\gamma)}}\\
	&= \up{\left(\kbordermatrix{
	    & 1 & 1 \\
	    1 & a & 0\\
	    1 & 0  & c 
	  };
	  \kbordermatrix{
	    & d & 1 \\
	    1 & d & 0\\
	    1 & 0  & 0 
        }\right)}
        \sqcap\kbordermatrix{
	    & 1 & 1 \\
	    1 & 1 & 1\\
	    1 & 1 & 1 
        };\kbordermatrix{
	    & d & 1 \\
	    1 & d & 1\\
	    1 & d & 1 
        }\\
        &= \renewcommand{\kbrdelim}{)^\uparrow}\kbordermatrix{
	    & d & 1 \\
	    1 & 0 & 0\\
	    1 & 0 & 0 
        }\renewcommand{\kbrdelim}{)}
        \sqcap\kbordermatrix{
	    & d & 1 \\
	    1 & d & 1\\
	    1 & d & 1 
        }\\
        &= \kbordermatrix{
	    & d & 1 \\
	    1 & 0 & 0\\
	    1 & 0 & 0
        }.
\end{align*}
For the right-hand side of Theorem~\ref{Th:RelHArrow}(2) we get
\begin{align*}
	\lefteqn{(\up{Q}_2\sqcap\Top_{(A,\alpha)(C,\gamma)};\Top_{(C,\gamma)(B,\beta')});\down{R}_2}\\
	&= \left(\kbordermatrix{
	    & 1 & 1 \\
	    1 & 1 & 0\\
	    1 & 0  & 1 
	  }\sqcap\kbordermatrix{
	    & d & 1 \\
	    1 & d & 1\\
	    1 & d & 1 
        };\kbordermatrix{
	    & 1 & 1 \\
	    d & d & d\\
	    1 & 1 & 1
        }\right); \kbordermatrix{
	    & d & 1 \\
	    1 & d & 0\\
	    1 & 0 & 0 
        }\\	
        &= \kbordermatrix{
	    & 1 & 1 \\
	    1 & 1 & 0\\
	    1 & 0  & 1 
	  }; \kbordermatrix{
	    & d & 1 \\
	    1 & d & 0\\
	    1 & 0 & 0 
        }\\	
        &= \kbordermatrix{
	    & d & 1 \\
	    1 & d & 0\\
	    1 & 0 & 0 
        }, 
\end{align*}
again verifying that equation from Theorem~\ref{Th:RelHArrow}(2) is not valid.

Arrow category use the axiom $\up{(Q;\down{R})}=\up{Q};\down{R}$ instead of Theorem~\ref{Th:RelHArrow}(2) modeling the situation where all relations use the same lattice of truth
values everywhere. Our next example will show that this equation is not valid, i.e., that the intersection with the composition of two universal relations in Theorem~\ref{Th:RelHArrow}(2)
cannot be dropped, even if all objects in question are uniform. Therefore, we define two relations $Q:(\{\ast\},\overline{1})\to(\{\ast\},\overline{1})$ 
and $R:(\{\ast\},\overline{1})\to(\{\ast\},\overline{d})$ by
\[
  Q=\kbordermatrix{
	    & 1\\
	    1\!\!\! & a
        },\qquad
  R=\kbordermatrix{
	    & d\\
	    1\!\!\! & d
        }.
\]
Please note that both objects $(\{\ast\},\overline{1})$ and $(\{\ast\},\overline{d})$ are uniform.
We obtain
\begin{align*}
	\up{(Q;\down{R})}
	&= \up{\left(\kbordermatrix{
	    & 1\\
	    1\!\!\! & a
        }; \kbordermatrix{
	    & d\\
	    1\!\!\! & d
        }\right)}\\
	&= \kbordermatrix{
	    & d\\
	    1\!\!\! & 0
        },\\
	\up{Q};\down{R}
	&= \kbordermatrix{
	    & 1\\
	    1\!\!\! & 1
        }; \kbordermatrix{
	    & d\\
	    1\!\!\! & d
        }\\
        &= \kbordermatrix{
	    & d\\
	    1\!\!\! & d
        },
\end{align*}
verifying that the equation $\up{(Q;\down{R})}=\up{Q};\down{R}$ is not valid. 

In Lemma~\ref{Lem:ArrowBasics}(6) we will show that Theorem~\ref{Th:RelHArrow}(3) is also valid if we replace the unit by any uniform object. Our last example will show that 
uniformity cannot be dropped, i.e., that $\up{Q}=\Top_{AA}$ for every ideal $Q:A\to A$ is not valid. We define
\[
Q=\kbordermatrix{
 & a & d\\
 a & a & 0\\
 d & 0 & 0 
},
\]
i.e., $Q:(A,\alpha')\to(A,\alpha')$ with $\alpha'(r)=a$ and $\alpha'(s)=d$. This relation is non-zero and an ideal since we have
\begin{align*}
	\lefteqn{\Top_{(A,\alpha')(A,\alpha')};Q;\Top_{(A,\alpha')(A,\alpha')}}\\
	&= \kbordermatrix{
	 & a & d\\
	 a & a & 0\\
	 d & 0 & d 
	};
	\kbordermatrix{
	 & a & d\\
	 a & a & 0\\
	 d & 0 & 0 
	};\kbordermatrix{
	 & a & d\\
	 a & a & 0\\
	 d & 0 & d 
	}\\
	&= \kbordermatrix{
	 & a & d\\
	 a & a & 0\\
	 d & 0 & 0 
	}\\
	&= Q.	
\end{align*}
On the other hand $\up{Q}=Q\ne\Top_{(A,\alpha')(A,\alpha')}$.

\section{Heyting allegories with Arrows}

We now take the properties of $\mathrm{Rel}({\mathcal H})$ listed in Theorem~\ref{Th:RelHArrow} and define Heyting allegories with arrows as a weakening of arrow categories.

\begin{definition}\label{Def:Arrow} 
Let $\kat{R}$ be a Heyting allegory with a unit and a pair of arrow operations. Then $\kat{R}$ is called a Heyting allegory with arrows iff the following conditions are satisfied:
\begin{enumerate}
	\item $\up{\id}_A=\id_A$ and $\up{\Top}_{AB}=\Top_{AB}$ for all objects $A$ and $B$.
	\item $\up{(Q;\down{R})}\sqcap\Top_{AB};\Top_{BC}=(\up{Q}\sqcap\Top_{AC};\Top_{CB});\down{R}$ 
    		for all relations $Q:A\to B$ and $R:B\to C$ and uniform objects $B$ and $C$.
    	\item If $Q\ne\Bot_{\unit\unit}$, then we have $\up{Q}=\Top_{\unit\unit}$.
\end{enumerate}
\end{definition}

Please note that the definition above differs from the definition of an arrow category. Arrow categories were introduced as an abstract theory for lattice valued relations
that are based on one fixed Heyting algebra. As a consequence, in an arrow category, all Heyting algebras $\ideals{A}{B}$ need to be isomorphic, which follows from 
uniformity as already mentioned above. The axiom $\Top_{AB}\neq\Bot_{AB}$ for all objects $A$ and $B$ of an arrow category avoids the existence of a null object. 
Intuitively, a null object corresponds to the empty set as an object. Since there are no elements the empty relation is the only relation on the empty set. This relation does 
not use any truth values (or exactly one truth value depending on the point of view), and, hence, may not use the same truth values than the relations on non-empty sets. 
Furthermore, the definition of an arrow category uses the axiom
$$ \up{(\trans{Q};\down{R})}=\trans{\up{Q}};\down{R} $$ 
for all $Q:B\to A$ and $R:B\to C$. This axiom can be separated into the two axioms
$$ \up{\trans{Q}}=\trans{\up{Q}}\mbox{ and } \up{(Q;\down{R})}=\up{Q};\down{R} $$ 
for all $Q:A\to B$ and $R:B\to C$. The latter version separates the relationship between the arrows and converse resp.\ composition into two axioms, which is more
convenient in our current context. The first axiom holds in all cases considered in this paper, and, is, therefore, included in the definition above. The second axiom has been 
replaced by the weaker version stated in (2).

By Theorem \ref{Th:RelHArrow} $\mathrm{Rel}({\mathcal H})$ is a Heyting allegory with arrows. Furthermore, $\mathrm{Rel}^u({\mathcal H})$ and ${\mathcal H}{\rm-Rel}$
are also Heyting algebra with arrows because both are full suballegories of $\mathrm{Rel}({\mathcal H})$ containing the unit by Lemma \ref{Lem:Unituniform} and the definition 
of a Heyting allegory with arrows is closed under suballegories due to the nature of the axioms. In addition, $\mathrm{Rel}^u({\mathcal H})$ is locally uniform by 
Lemma \ref{Lem:RelHobjectuniform} and ${\mathcal H}{\rm-Rel}$ is uniform because 
\begin{align*}
	(\Top_{(A,\overline{1})(B,\overline{1})};\Top_{(B,\overline{1})(C,\overline{1})})(a,c)
	&= \bigvee_{b\in B}\Top_{(A,\overline{1})(B,\overline{1})}(a,b)\wedge\Top_{(B,\overline{1})(C,\overline{1})}(b,c)\\
	&= \bigvee_{b\in B}1\wedge 1\\
	&= 1\\
	&= \Top_{(A,\overline{1})(C,\overline{1})}(a,c).
\end{align*}

Next we want to verify that the unit is always uniform.

\begin{lemma}\label{Lem:Unituniform}
Let $\kat{R}$ be a Heyting category with arrows. Then the unit $unit$ is uniform.
\end{lemma}

\begin{proof}
Assume $v:\unit\to\unit$ is non-zero. Then by Definition \ref{Def:Arrow}(3) we have $\up{v}=\Top_{\unit\unit}$ and obtain
$ \up{v};\Top_{\unit\unit}= \Top_{\unit\unit};\Top_{\unit\unit}=\Top_{\unit\unit};\id_{\unit}=\Top_{\unit\unit}$. 
\end{proof}

The following lemma lists some basic properties of Heyting allegory with arrows.

\begin{lemma}\label{Lem:ArrowBasics}
Let $\kat{A}$ be a Heyting allegory with arrows and $Q,R:A\to B$. Then we have:
\begin{enumerate}
	\item $Q^{\downarrow\uparrow}=\down{Q}$ and $Q^{\uparrow\downarrow}=\up{Q}$.
	\item $\up{Q}=Q$ iff $\down{Q}=Q$.
	\item $\down{\rule{0pt}{6pt}}$ is an interior and $\up{\rule{0pt}{6pt}}$ a closure operation, i.e., we have
		$\down{Q}\sqsubseteq Q$, $Q^{\downarrow\downarrow}=\down{Q}$ and $Q\sqsubseteq\up{Q}$, $Q^{\uparrow\uparrow}=\up{Q}$.
	\item $\up{\Top}_{AB}=\down{\Top}_{AB}=\Top_{AB}$, $\up{\id}_A=\down{\id}_A=\id_A$, and $\up{\Bot}_{AB}=\down{\Bot}_{AB}=\Bot_{AB}$.
	\item $\up{(\down{Q};R)}\sqcap\Top_{AB};\Top_{BC}=\down{Q};(\up{R}\sqcap\Top_{BA};\Top_{AC})$ 
    		for all relations $Q:A\to B$ and $R:B\to C$ and uniform objects $A$ and $B$.
	\item If $Q:A\to B$ is a non-zero ideal and $A$ and $B$ are uniform, then $\up{Q}=\Top_{AB}$.
\end{enumerate}
\end{lemma}

\begin{proof}
\begin{enumerate}
	\item The first equation is shown by
		\begin{align*}
			Q^{\downarrow\uparrow}
			&= \up{(\Top_{AB}\sqcap\down{Q})}\\
			&= \up{\Top}_{AB}\sqcap\down{Q} && \mbox{Definition \ref{Def:PairArrows}(3)}\\
			&= \Top_{AB}\sqcap\down{Q} && \mbox{Definition \ref{Def:Arrow}(1)}\\
			&= \down{Q}.
		\end{align*}
		Using the first equation we obtain
		\begin{align*}
			Q^{\uparrow\downarrow}
			&= Q^{\uparrow\downarrow\uparrow} && \mbox{first equation}\\
			&= \up{Q}. && \mbox{Lemma \ref{Lem:PairArrowsBasics}(3)}
		\end{align*}
	\item Suppose $\up{Q}=Q$. Then by using (1)  we obtain $\down{Q}=Q^{\uparrow\downarrow}=\up{Q}=Q$.
		Similarly, if $\down{Q}=Q$, we get $\up{Q}=Q^{\downarrow\uparrow}=\down{Q}=Q$.
	\item We have $\down{Q}=Q^{\downarrow\uparrow}\sqsubseteq Q$ by using (1) and Lemma \ref{Lem:PairArrowsBasics}(1). From Lemma \ref{Lem:PairArrowsBasics}(3) and (1) we obtain 
		$\down{Q}=Q^{\downarrow\uparrow\downarrow}=Q^{\downarrow\downarrow}$. The properties for $\up{\rule{0pt}{6pt}}$
		are shown analogously.
    	\item This follows from (2) and Definition \ref{Def:Arrow}(1) and Lemma \ref{Lem:PairArrowsBasics}(5).
    	\item We immediately compute
    		\begin{align*}
    			\lefteqn{\up{(\down{Q};R)}\sqcap\Top_{AB};\Top_{BC}}\\
    			&= \trans{(\trans{\up{(\down{Q};R)}}\sqcap\Top_{CB};\Top_{BA})}\\
    			&= \trans{(\up{(\trans{R};\down{\trans{Q}})}\sqcap\Top_{CB};\Top_{BA})} && \mbox{Definition \ref{Def:PairArrows}(4) and (10)}\\
    			&= \trans{((\up{\trans{R}}\sqcap\Top_{CA};\Top_{AB});\down{\trans{Q}})} && \mbox{Definition \ref{Def:Arrow}(2) with $A, B$ uniform}\\
    			&= \trans{((\trans{\up{R}}\sqcap\Top_{CA};\Top_{AB});\trans{\down{Q}})} && \mbox{Definition \ref{Def:PairArrows}(4) and (10)}\\
    			&= \down{Q};(\up{R}\sqcap\Top_{BA};\Top_{AC}).
    		\end{align*}
	\item Suppose $Q\ne\Bot_{AB}$. First we want to show that $\Top_{\unit A};Q;\Top_{B\unit}$ is non-zero. Therefore, assume $\Top_{\unit A};Q;\Top_{B\unit}=\Bot_{\unit\unit}$.
		Then we obtain
		\begin{align*}
			Q
			&= \Top_{AA};Q;\Top_{BB} && \mbox{$Q$ ideal}\\
			&= \Top_{A\unit};\Top_{\unit A};Q;\Top_{B\unit};\Top_{\unit B} && \mbox{Lemma \ref{Lem:TopUnit}}\\	
			&= \Top_{A\unit};\Bot_{\unit\unit};\Top_{\unit B} && \mbox{assumption}\\	
			&= \Bot_{AB},
		\end{align*}
		a contradiction. Therefore, $\Top_{\unit A};Q;\Top_{B\unit}\ne\Bot_{\unit\unit}$. Now, the following computation
		\begin{align*}
			\up{Q}
			&= \up{(\Top_{AA};Q;\Top_{BB})} && \mbox{$Q$ ideal}\\
			&= \up{(\Top_{AA};Q;\Top_{B\unit};\Top_{\unit B})} && \mbox{Lemma \ref{Lem:TopUnit}}\\			
			&= \up{(\Top_{AA};Q;\Top_{B\unit};\down{\Top}_{\unit B})} && \mbox{by (4)}\\			
			&= \up{(\Top_{AA};Q;\Top_{B\unit};\down{\Top}_{\unit B})}\sqcap\Top_{A\unit};\Top_{\unit B} && \mbox{Lemma \ref{Lem:TopUnit}}\\
			&= \up{(\Top_{AA};Q;\Top_{B\unit}\sqcap\Top_{AB};\Top_{B\unit})};\down{\Top}_{\unit B} && \mbox{Definition \ref{Def:Arrow}(2), $\unit, B$ uniform}\\
			&= \up{((\Top_{AA};Q;\Top_{BB}\sqcap\Top_{AB});\Top_{B\unit})};\down{\Top}_{\unit B} && \mbox{Lemma \ref{Lem:Basics}(3)}\\
			&= \up{(\Top_{AA};Q;\Top_{BB};\Top_{B\unit})};\Top_{\unit B} && \mbox{by (4)}\\
			&= \up{(\Top_{AA};Q;\Top_{B\unit})};\Top_{\unit B} && \mbox{Lemma \ref{Lem:Basics}(1)}\\			
			&= \up{(\Top_{A\unit};\Top_{\unit A};Q;\Top_{B\unit})};\Top_{\unit B} && \mbox{Lemma \ref{Lem:TopUnit}}\\			
			&= \up{(\down{\Top}_{A\unit};\Top_{\unit A};Q;\Top_{B\unit})};\Top_{\unit B} && \mbox{by (4)}\\
			&= (\up{(\down{\Top}_{A\unit};\Top_{\unit A};Q;\Top_{B\unit})}\sqcap\Top_{A\unit};\Top_{\unit\unit});\Top_{\unit B} && \mbox{Lemma \ref{Lem:Basics}(1)}\\
			&= \down{\Top}_{A\unit};\up{(\Top_{\unit A};Q;\Top_{B\unit}\sqcap\Top_{\unit A};\Top_{A\unit})};\Top_{\unit B} && \mbox{by (11) with $\unit, A$ uniform}\\
			&= \down{\Top}_{A\unit};\up{(\Top_{\unit A};(Q;\Top_{B\unit}\sqcap\Top_{AA};\Top_{A\unit}))};\Top_{\unit B} && \mbox{Lemma \ref{Lem:Basics}(3)}\\
			&= \Top_{A\unit};\up{(\Top_{\unit A};Q;\Top_{B\unit})};\Top_{\unit B} && \mbox{by (4) and Lemma \ref{Lem:Basics}(1)}\\
			&= \Top_{A\unit};\Top_{\unit\unit};\Top_{\unit B} && \mbox{Definition \ref{Def:Arrow}(3)}\\
			&= \Top_{AB} && \mbox{$\Top_{\unit\unit}=\id_\unit$ and Lemma \ref{Lem:TopUnit}}
		\end{align*}
		that verifies the assertion.
\end{enumerate}
\end{proof}

The following theorem characterizes $\mathrm{Rel}^u({\mathcal H})$ abstractly.

\begin{theorem}\label{Th:LUArrow}
Let $\kat{R}$ be a Heyting allegory with a unit and a pair of arrow operations. Then the following statements are equivalent:
\begin{enumerate}
	\item $\kat{R}$ is a locally uniform Heyting allegory with arrows.
	\item The following properties are satisfied:
		\begin{enumerate}
			\item $\up{\id}_A=\id_A$ for all objects $A$.
			\item $\up{(Q;\down{R})}\sqcap\Top_{AB};\Top_{BC}=(\up{Q}\sqcap\Top_{AC};\Top_{CB});\down{R}$ for all relations $Q:A\to B$ and $R:B\to C$.
    			\item If $Q\ne\Bot_{AB}$ is an ideal, then we have $\up{Q}=\Top_{AB}$.
		\end{enumerate}
\end{enumerate}
\end{theorem}

\begin{proof}
\begin{description}
	\item[$(1)\Rightarrow(2)${\rm :}] Property (a) is part of Definition \ref{Def:Arrow}(1). Property (b) follows immediately from Definition \ref{Def:Arrow}(2) because all objects
		in $\kat{R}$ are uniform. Finally, Property (c) is shown by Lemma \ref{Lem:ArrowBasics}(6), again, because all objects in $\kat{R}$ are uniform.
	\item[$(2)\Rightarrow(1)${\rm :}] In order to show that $\kat{R}$ is a Heyting allegory with arrows it is sufficient to show that $\up{\Top}_{AB}=\Top_{AB}$ since all other properties
		are just special cases of (a)-(c). If $\Top_{AB}=\Bot_{AB}$, then we have $\up{\Top}_{AB}=\up{\Bot}_{AB}=\Bot_{AB}=\Top_{AB}$ by using Lemma \ref{Lem:PairArrowsBasics}(5). 
		If $\Top_{AB}\ne\Bot_{AB}$, then the assertion follows immediately from (c). It remains to show that $\kat{R}$ is locally uniform. Suppose $v:\unit\to A$ is non-zero. 
		First we want to show that $v;\Top_{A\unit}$ is a non-zero ideal. It is an ideal because every relation on the $\unit$ is an ideal. Assume $v;\Top_{A\unit}=\Bot_{\unit\unit}$. Then
		we get
		\begin{align*}
			v
			&\sqsubseteq v;\Top_{AA}\\
			&= v;\Top_{A\unit};\Top_{\unit A} && \mbox{Lemma \ref{Lem:TopUnit}}\\
			&= \Bot_{\unit\unit};\Top_{\unit A} && \mbox{assumption}\\
			&= \Bot_{\unit\unit},
		\end{align*}	
		a contradiction. Therefore, $v;\Top_{A\unit}\ne\Bot_{\unit\unit}$ and (c) implies that $\up{(v;\Top_{A\unit})}=\Top_{\unit\unit}$. Now, compute
		\begin{align*}
			\up{v};\Top_{A\unit}
			&= \up{(v\sqcap\Top_{\unit\unit};\Top_{\unit A})};\down{\Top}_{A\unit} && \mbox{Lemma \ref{Lem:ArrowBasics}(4) and $\Top_{\unit\unit}=\id_\unit$}\\
			&= \up{(v;\down{\Top}_{A\unit})}\sqcap\Top_{\unit A};\Top_{A\unit} && \mbox{by (b)}\\
			&= \up{(v;\Top_{A\unit})}\sqcap\Top_{\unit A};\Top_{A\unit} && \mbox{Lemma \ref{Lem:ArrowBasics}(4)}\\
			&= \Top_{\unit\unit}\sqcap\Top_{\unit A};\Top_{A\unit} && \mbox{see above}\\
			&= \Top_{\unit A};\Top_{A\unit},
		\end{align*}
		verifying that $A$ is uniform.
\end{description}
\end{proof}

The following lemma relates locally uniform Heyting allegory with arrows with arrow categories. Please note that the property $\Top_{AB}\ne\Bot_{AB}$ is part of the definition
of an arrow allegory and, therefore, needs to be added to the following theorem. The reason for including this property is convenience, i.e., avoiding the corresponding
case distinction in each proposition.

\begin{theorem}
Let $\kat{R}$ be a locally uniform Heyting allegory with arrows so that $\Top_{AB}\ne\Bot_{AB}$ for every $A$ and $B$. Then the following statements are equivalent:
\begin{enumerate}
	\item $\kat{R}$ is a uniform.
	\item $\up{(Q;\down{R})}=\up{Q};\down{R}$ for all relations $Q:A\to B$ and $R:B\to C$.
\end{enumerate}
\end{theorem}

\begin{proof}
\begin{description}
	\item[$(1)\Rightarrow(2)${\rm :}] This follows immediately since Theorem \ref{Th:LUArrow}(2b) implies (2) if $\kat{R}$ is uniform.
	\item[$(2)\Rightarrow(1)${\rm :}] First we want to show that $\up{(\Top_{AB};\Top_{BA})}=\Top_{AA}$. Since 
		$\Top_{AB};\Top_{BA}$ is an ideal this follows immediately from $\Top_{AB};\Top_{BA}\neq\Bot_{AA}$ using Lemma \ref{Th:LUArrow}(2c).
		To verify the latter condition, suppose $\Top_{AB};\Top_{BA}=\Bot_{AA}$. Then using Lemma \ref{Lem:Basics}(4) we obtain 
		$\Top_{AB}\sqsubseteq\Top_{AB};\Top_{BA};\Top_{AB}=\Bot_{AA};\Top_{AB}=\Bot_{AB}$, a contradiction to the assumption
		of this lemma. Now we compute
		\begin{align*}
			\Top_{AB};\Top_{BA}
			&= \up{\Top}_{AB};\down{\Top}_{BA}  && \mbox{Lemma \ref{Lem:ArrowBasics}(4)}\\
			&= \up{(\Top_{AB};\down{\Top}_{BA})} && \mbox{by (2)}\\
			&= \up{(\Top_{AB};\Top_{BA})}  && \mbox{Lemma \ref{Lem:ArrowBasics}(4)}\\
			&= \Top_{AA}. && \mbox{see above}
		\end{align*}
	By Lemma \ref{Lem:uniform} it follows that $\kat{A}$ is uniform.
\end{description}
\end{proof}

\section{Conclusion and Future Work}

In this paper we generalized the theory of arrow allegories to the cases where relations between different objects may use different Heyting algebras as  truth values. This led
to the notion of locally uniform Heyting allegories with arrows. We also provided a suitable set of axioms (Theorem~\ref{Th:LUArrow}) for this class of allegories. Furthermore,
we generalized the concept even further to Heyting allegories with arrows which allow relations to use a different lattice of truth values for every pair in the relation. 

Certain applications of fuzzy relations require switching form one lattice of truth values to another. Replacing the truth values used by a different lattice is sometimes called a change of base.
Concrete examples utilizing a change of base include higher-order fuzzy controllers or, more generally, complex decision making and network analysis. In terms of arrow allegories a change of base
requires suitable operations between two (or more) different arrow allegories This has been studied in \cite{Winter2014, Winter2021}. Using the framework of Heyting allegories with arrows 
in this context will have two advantages. It is possible to characterize objects that are based on the same set of elements but use a different lattice of truth values abstractly. This together 
with an appropriate definition of the operations mentioned above should allow to define and apply a change of base within a single allegory. i.e., within a framework based on a single 
categorical/algebraic structure. We will investigate these topics in future work.

\end{document}